\documentclass[11pt,a4paper]{article}
\usepackage[utf8]{inputenc}
\usepackage{amssymb}
\usepackage{amsmath}
\usepackage{amsthm}
\usepackage{amsfonts}
\usepackage[left=2.45cm, top=2.45cm,bottom=2.45cm,right=2.45cm]{geometry}

\newtheorem{theorem}{Theorem}
\newtheorem{lemma}[theorem]{Lemma}

\newtheorem{proposition}[theorem]{Proposition}
\newtheorem{conjecture}[theorem]{Conjecture}

\theoremstyle{definition}
\newtheorem{remark}{Remark}

\newcommand{\N}{\mathbb{N}}

\newcommand{\IE}{\mathbb{E}}

\newcommand{\IN}{\mathbb{N}}
\newcommand{\IP}{\mathbb{P}}

\newcommand{\IR}{\mathbb{R}}

\newcommand{\vol}{\mathrm{vol}}
\newcommand{\dd}{\mathrm{d}}
\newcommand{\bp}{\mathbb{B}_p^n}
\newcommand{\binf}{\mathbb{B}_{\infty}^n}
\newcommand{\bpb}{\mathbb{B}_{p,\beta}^n}
\newcommand{\btb}{\mathbb{B}_{2,\beta}^n}
\newcommand{\btt}{\mathbb{B}_{2,2}^n}
\newcommand{\binfb}{\mathbb{B}_{\infty,\beta}^n}
\newcommand{\Dp}{\mathbb{D}_p^n}
\newcommand{\Dq}{\mathbb{D}_q^n}
\newcommand{\Dinf}{\mathbb{D}_{\infty}^n}
\newcommand{\Dpb}{\mathbb{D}_{p,\beta}^n}

\newcommand{\Dqb}{\mathbb{D}_{q,\beta}^n}

\newcommand{\tc}{t}

\title{Critical Intersections of unit balls}
\author{Mathias Sonnleitner\footnote{
	 Faculty of Computer Science and Mathematics, University of Passau, 94032 Passau, Germany\\
	 \noindent {\it E-mail:} \texttt{mathias.sonnleitner@uni-passau.de}}
	 and Christoph Thäle\footnote{
	\noindent \textsc{Christoph Th\"ale:} Faculty of Mathematics,
			University of Bochum, 44780 Bochum, Germany\\
	\noindent {\it E-mail:} \texttt{christoph.thaele@rub.de}}}

\begin{document}

\title{\bfseries A note on critical intersections\\ of classical and Schatten $p$-balls}

\author{Mathias Sonnleitner\footnotemark[1]\;\; and Christoph Th\"ale\footnotemark[2]}

\date{}
\renewcommand{\thefootnote}{\fnsymbol{footnote}}
\footnotetext[1]{Faculty of Computer Science and Mathematics, University of Passau, Germany. Email: mathias.sonnleitner@uni-passau.de}

\footnotetext[2]{
	Faculty of Mathematics, Ruhr University Bochum, Germany. Email: christoph.thaele@rub.de}

\maketitle

\begin{abstract}
	\noindent  The purpose of this note is to study the asymptotic volume of intersections of unit balls associated with two norms in $\mathbb{R}^n$ as their dimension $n$ tends to infinity. A general framework is provided and then specialized to the following cases. For classical $\ell_p^n$-balls the focus lies on the case $p=\infty$, which has previously not been studied in the literature. As far as Schatten $p$-balls are considered, we concentrate on the cases $p=2$ and $p=\infty$. In both situations we uncover an unconventional limiting behavior.
\end{abstract}

\section{Introduction}

One way to approach the asymptotic distribution of volume in a sequence of convex bodies in high dimensions is to study the limiting behavior of the intersection volumes with another sequence of gauge or test bodies. In particular, this applies to the unit balls of finite dimensional normed spaces. To give a concrete example, let us consider the $\ell_p^n$-balls for $1\leq p\le \infty$, which are defined as
$$
\bp
:=\begin{cases}
\{x=(x_1,\ldots,x_n)\in\IR^n\colon \sum_{j=1}^{n}|x_j|^p\le 1\} &\colon p<\infty\\
\{x=(x_1,\ldots,x_n)\in\IR^n\colon \max_{j=1,\ldots,n}|x_j|\le 1\} &\colon p=\infty.
\end{cases}
$$
We also define their volume-normalized versions
$$
\Dp=\vol_n(\bp)^{-1/n}\bp
$$
with $\vol_n(\bp)^{1/n}$ being the volume radius. The following result has been shown in the non-critical case $t\neq \tc_{p,q}$ by Schechtman and Schmuckenschläger in \cite{SS91} and in the critical case $t=\tc_{p,q}$ by Schmuckenschläger in \cite{Sch98} (for $p=\infty$) and \cite{Sch01} (for $p<\infty$), both building on a probabilistic approach to $\ell_p$-balls independently developed in \cite{RR91} and \cite{SZ90}. A streamlined modern proof together with a multivariate extension has been given in \cite{KPT19}. 
There are also closely related further results in this direction with respect to other families of norms and convex sets. For example, there are results for the simplex \cite{BKP+20}, mixed-norm balls \cite{JKP22}, $\ell_p$-ellipsoids \cite{JP21}, Orlicz balls \cite{KP21} and Lorentz balls \cite{KPS23}. We also refer to the survey article \cite{PTT19}.

\begin{proposition}[Schechtman/Schmuckenschl\"ager] \label{thm:p-ball}
Let $1\leq p\le \infty$ and $1\le q<\infty$ be such that $p\neq q$. Then there exists a constant $\tc_{p,q}\in (0,\infty)$ only depending on $p$ and $q$ such that 
\[
\lim_{n\to\infty}\vol_n(\Dp \cap t \Dq)
=
\begin{cases}
	0 &\colon t<\tc_{p,q}\\
	1/2 &\colon t=\tc_{p,q}\\
	1 &\colon t>\tc_{p,q}.\\
\end{cases}
\]
\end{proposition}

\begin{remark}
We remark that the constant $\tc_{p,q}$ is known explicitly, but rather involved. We refer the reader to the aforementioned papers for details.
\end{remark}

It seems that the case $q=\infty$ in Proposition \ref{thm:p-ball} has been overlooked in the existing literature. Our first result closes this gap.

\begin{theorem} \label{thm:pball}
	Let $1\leq p< \infty$. For any $t>0$ it holds that
	$$\lim_{n\to\infty}\vol_n(\Dp \cap t \Dinf)=0.
	$$
	Moreover,
	\[
	\lim_{n\to\infty}\vol_n(\Dp \cap t(\log n)^{1/p} \Dinf)
	=
	\begin{cases}
		0 &\colon t<\tc_{p,\infty}\\
		0 &\colon t=\tc_{p,\infty}\text{ and }p>1\\
		e^{-1} &\colon t=\tc_{1,\infty}\text{ and }p=1\\
		1 &\colon t>\tc_{p,\infty},\\
	\end{cases}
	\]
	where $\tc_{p,\infty}=e^{-1/p}\Gamma(1+1/p)^{-1}$.
\end{theorem}

Of particular interest for us in this note are the non-commutative analogues of the $\ell_p^n$-balls considered above. To introduce them, fix $n\in\IN$ and $\beta\in\{1,2,4\}$, and define the space $H_n(\mathbb{F}_{\beta})$ of all self-adjoint $n\times n$ matrices over $\mathbb{F}_{\beta}$, where $\mathbb{F}_{1}=\IR$  is the field of real numbers, $\mathbb{F}_2=\mathbb{C}$ the field of complex numbers and $\mathbb{F}_4=\mathbb{H}$ is the skew-field of quaternions. For $1\le p\le \infty$ we define the unit ball with respect to the Schatten-$p$-norm by
\[
\bpb
:=\begin{cases}
\Big\{A\in H_n(\mathbb{F}_{\beta})\colon \sum_{j=1}^{n}|\lambda_j(A)|^p\le 1\Big\} &\colon p<\infty\\[2mm]
\Big\{A\in H_n(\mathbb{F}_{\beta})\colon \max_{j=1,\ldots,n}|\lambda_j(A)|\le 1\Big\} &\colon p=\infty,
\end{cases}
\]
where $\lambda_1(A),\ldots,\lambda_n(A)$ denote the eigenvalues of $A$. Since $H_n(\mathbb{F}_{\beta})$ is a subspace of $\mathbb{F}_{\beta}^{n^2}\equiv \IR^{\beta n^2}$ of dimension 
\begin{align}\label{eq:dn}
d_n:=\beta\frac{n(n-1)}{2}+n,
\end{align}
the convex bodies $\bpb$ have dimension $d_n$ as well and we denote their $d_n$-dimensional Lebesgue measure by $\vol_{\beta,n}(\bpb)$. We can again consider the normalized Schatten $p$-balls
\[
\Dpb
:=\vol_{\beta,n}(\bpb)^{-1/d_n}\bpb.
\]
The limiting behavior of the intersection volumes of these balls has been determined in \cite[Theorem 5.1]{KPT20}, which states that for any $t>0$ and for all $1\le p,q<\infty$,
\[
\lim_{n\to\infty}\vol_{\beta,n}(\Dpb\cap t \Dqb)
=
\begin{cases}
	0 &\colon t<\tc_{p,q,\beta}\\
	1 &\colon t>\tc_{p,q,\beta},\\
\end{cases}
\]
where
\[
\tc_{p,q,\beta}:=e^{{1\over 2}(1/p-1/q)}\Big(\frac{2p}{p+q}\Big)^{1/q}.
\]

In view of the behavior of $\ell_p^n$ balls described by Proposition \ref{thm:p-ball}, it is natural to ask what happens at the threshold $t=\tc_{p,q,\beta}$. We recall that the corresponding result for $\ell_p^n$-balls was based on a central limit theorem for the $q$-norm of a uniformly distributed random point in $\bp$. However, for Schatten $p$-balls a result of this type is not available. Our next theorem can be considered as a first attempt to close this gap. More precisely, we are able to prove the following behavior for the critical intersections in the special case $\beta=2$, $p=2$ and $q=\infty$.

\begin{theorem} \label{thm:schatten-crit}
For $\beta=2$ and any $t>0$ it holds that
\[
\lim_{n\to\infty}\vol_{\beta,n}(\mathbb{D}_{2,\beta}^n\cap t \mathbb{D}_{\infty,\beta}^n)
\begin{cases}
	0 &\colon t<\tc_{2,\infty,\beta}\\
	F_{\beta}(0)^2 &\colon t=\tc_{2,\infty,\beta}\\
	1 &\colon t>\tc_{2,\infty,\beta},\\
\end{cases}
\]
where $F_{\beta}$ is the distribution function of the Tracy-Widom($\beta$) distribution (see, e.g., \cite[Section 3.1]{AGZ10}) and $\tc_{2,\infty,2}=e^{1/4}$.
\end{theorem}

\begin{remark}
The exact value of $F_{\beta}(0)$ is unknown but by numerical approximation we get the values summarized in the following table.
\begin{center}
	\begin{tabular}{c|c|c|c}
& $\beta=1$ & $\beta=2$ & $\beta=4$\\
\hline
$F_{\beta}(0)$ & $0.831908$ & $0.969373$ & $0.998574$\\
$F_{\beta}(0)^2$ & $0.692071$ & $0.939684$ & $0.997150$\\
\end{tabular}
\end{center}
\end{remark}

\begin{remark}
	We have $\tc_{2,\infty,2}=e^{1/4}=\lim_{q\to\infty}\tc_{2,q,2}$. It seems reasonable to believe that the non-critical case of Theorem~\ref{thm:schatten-crit} extends to $p,\beta\neq 2$ with $\tc_{p,\infty,\beta}=\lim_{q\to\infty}\tc_{p,q,\beta}=e^{1/2p}$.
\end{remark}

The limitation of Theorem \ref{thm:schatten-crit} to $p=2$, $q=\infty$ and $\beta=2$ has two major reasons. The first one can be described as follows:
\begin{itemize}
	\item[1.] Only for $p=2$ and $p=\infty$ the volume of $\bpb$ is known exactly. Although this is not strictly needed, what is required is a precise asymptotic analysis of the volume of $\bpb$ on a scale which is finer than that considered in \cite{KPT20}. It seems, however, that the technique used in \cite{KPT20} cannot be refined to give a result of appropriate quality.
\end{itemize}
The restriction to the case $p=2$ and $q=\infty$ can now be used to give some insight about the Tracy-Widom distribution and why it appears in our context. We briefly explain the connection of our problem to a classical question in random matrix theory. For that purpose consider for $\beta\in\{1,2,4\}$ the eigenvalues $X_i=X_i(\beta)$ of a  random matrix from a standard Gaussian $\beta$-ensemble. More precisely, this is the matrix ensemble GOE (Gaussian orthogonal ensemble) if $\beta=1$, GUE (Gaussian unitary ensemble) if $\beta=2$ and GSE (Gaussian symplectic ensemble) if $\beta=4$.  It is well-known that the largest eigenvalue satisfies
\begin{align}\label{eq:TracyWidom}
\lim_{n\to\infty}\IP[\sqrt{2}n^{2/3}(n^{-1/2}\max X_i-\sqrt{2})\le x] = F_{\beta}(x),
\end{align}
where $F_{\beta}$ is the Tracy-Widom($\beta$) distribution, see \cite[Section 3.1]{AGZ10} or for this formulation \cite{Bor10}. 
In order to derive Theorem~\ref{thm:schatten-crit} we need the limit law \eqref{eq:TracyWidom} or more precisely a limit theorem for $\max|X_i|$, for which it suffices to have a joint limit theorem for the largest and the smallest eigenvalue.  At this point, we can now explain the limitation of Theorem \ref{thm:schatten-crit} to $\beta=2$:
\begin{itemize}
\item[2.] The limitation to the case $\beta=2$ has its origin in the fact that only in this case it is known from \cite[Corollary 1]{BDN10} (see also \cite{Bor10} for more precise asymptotics) that the fluctuations of the smallest and the largest eigenvalue of random matrices from a Gaussian $\beta$-ensemble are asymptotically independent.
\end{itemize}
To be more precise about this second point, let us denote for $n\in\N$ and $\beta\in\{1,2,4\}$ the centred and rescaled smallest and largest eigenvalues of a random matrices from a $\beta$-ensemble by
\[
X_{\beta,n}^{\min}:=\sqrt{2}n^{1/6}(\min X_i +\sqrt{2n})
\qquad\text{and}\qquad
X_{\beta,n}^{\max}:=\sqrt{2}n^{1/6}(\max X_i -\sqrt{2n}),
\]
respectively. It seems reasonable to conjecture that these two random variables are asymptotically independent in the following sense. 

\begin{conjecture}\label{conjecture}
For $\beta\in\{1,2,4\}$ and $x,y\in\IR$ it holds that
\[
\lim_{n\to\infty}\Big(\IP[X_{\beta,n}^{\min}\le x, X_{\beta,n}^{\max}\le y]
-\IP[X_{\beta,n}^{\min}\le x]\IP[ X_{\beta,n}^{\max}\le y]\Big)=0.
\]
\end{conjecture}

As already mentioned above, \cite[Corollary 1]{BDN10} shows that Conjecture \ref{conjecture} is true for $\beta=2$. To the best of our knowledge, for $\beta\in \{1,4\}$ the conjecture is still open. This is also discussed after Proposition 1 in \cite{JL15}. A positive answer would immediately imply that Theorem \ref{thm:schatten-crit} holds true for all $\beta\in\{1,2,4\}$.

\medspace

The remaining parts of this note are structured as follows. In Section \ref{sec:Prelim} we develop a general framework to asymptotic intersection volumes, which might be of independent interest. The proof of Theorem \ref{thm:pball} is the content of Section \ref{sec:lp}, while we prove Theorem \ref{thm:schatten-crit} in the final Section \ref{sec:Schatten}.

\section{A general framework}\label{sec:Prelim}

In order to give a unified presentation of asymptotics of intersection volumes of unit balls, we choose to take a look at the general picture which we apply to $\ell_p^n$- and Schatten $p$-balls in the two forthcoming sections.

Fix a dimension $n\in\N$, and let $\|\,\cdot\,\|_{X^n}$ and $\|\,\cdot\,\|_{Y^n}$ be two norms on $\IR^n$. The closed unit balls generated by these two norms will be denoted by $B_X^n$ and $B_Y^n$ and we write $D_X^n$ and $D_Y^n$ for their volume-normalized versions, respectively. If $Z$ is uniformly distributed in $B_X^n$, then $\vol(B_X^n)^{-1/n}Z$ is uniformly distributed in $D_X^n$ and, for any $t>0$, positive homogenity implies
\begin{align*}
\vol_n(D_X^n\cap t D_Y^n)
&=\IP\big[ \vol(B_X)^{-1/n}Z \in t\, \vol_n(B_Y^n)^{-1/n}B_Y^n\big]\\
&=\IP\Big[ \|Z\|_{Y^n} \le t\frac{\vol_n(B_X^n)^{1/n}}{\vol_n(B_Y^n)^{1/n}}\Big].
\end{align*}
Now, suppose that there are constants $c_X,c_Y>0$ and positive sequences $(a_n)_{n\in\N},(b_n)_{n\in\N}$ such that
\[
\vol_n(B_X^n)^{1/n}\sim c_X a_n\qquad \text{and}\qquad
\vol_n(B_Y^n)^{1/n}\sim c_Y b_n
\]
holds as $n\to\infty$, where for two sequences $(A_n)_{n\in\N}$ and $(A_n')_{n\in\N}$ we write $A_n\sim A_n'$, provided that $A_n/A_n'\to 1$, as $n\to\infty$. From a weak law of large numbers for $\|Z\|_{Y^n}$ one can now deduce the behavior of the asymptotic intersection volumes of $D_X^n$ and $D_Y^n$ in the non-critical cases.

\begin{lemma} \label{lem:non-crit}
If the convergence
\[
\|Z\|_{Y^n}\frac{b_n}{a_n}\to c,\qquad\text{as }n\to\infty,
\]
holds in probability for some $c\in [0,\infty]$, then for any $t>0$,
\[
\lim_{n\to\infty}\vol_n(D_X^n\cap t D_Y^n)
=
\begin{cases}
	0 &\colon t<t_{X,Y}\\
	1 &\colon t>t_{X,Y},\\
\end{cases}
\]
where $t_{X,Y}:=c\frac{c_Y}{c_X}$.
\end{lemma}
\begin{proof}
	Let $t>0$. Then
\[
\IP\Big[ \|Z\|_{Y^n} \le t\frac{\vol_n(B_X^n)^{1/n}}{\vol_n(B_Y^n)^{1/n}}\Big]
=\IP\Big[ \|Z\|_{Y^n}\frac{b_n}{a_n} \le t\frac{\vol_n(B_X^n)^{1/n}a_n^{-1}}{\vol_n(B_Y^n)^{1/n}b_n^{-1}}\Big]
\]
The random variable $\|Z\|_{Y^n}\frac{b_n}{a_n}$ on the left-hand side tends in probability to $c$ by assumption. The expression on the right-hand side deterministically converges to $t\frac{c_X}{c_Y}$. This gives the result.
\end{proof}

In the critical case we need information on the fluctuations of the $\|\,\cdot\,\|_{Y^n}$-norm of $Z$ and more precise asymptotics for volume radii. In particular, let us assume that
\begin{equation}\label{eq:AsymVolRadius}
\frac{\vol_n(B_X^n)^{1/n}a_n^{-1}}{\vol_n(B_Y^n)^{1/n}b_n^{-1}}
=\frac{c_X}{c_Y}(1+e_n),
\end{equation}
where $(e_n)_{n\in\N}$ is a sequence satisfying $e_n\to 0$ as $n\to\infty$. Under this assumption we have the following result.

\begin{lemma} \label{lem:crit}
Suppose that the convergence
\[
c_n\Big(\|Z\|_{Y^n}\frac{b_n}{a_n}-c\Big)\to R\qquad\text{as }n\to\infty
\]
holds in distribution for some $c>0$, a positive sequence $(c_n)_{n\in\N}$ satisfying $c_n\to \infty$ and with a random variable $R$ having a continuous distribution function. Additionally assume that the limit $\lim_{n\to\infty}e_n c_n$ exists in the extended real line $\IR\cup \{-\infty,+\infty\}$. Then, for $t_{X,Y}$ as in Lemma~\ref{lem:non-crit},
\[
\lim_{n\to\infty}\vol_n(D_X^n\cap t_{X,Y} D_Y^n)
= \IP[R\le c\lim_{n\to\infty}e_n c_n],
\]
where we use the convention that $\IP[R\leq-\infty]=0$ and $\IP[R\leq+\infty]=1$.
\end{lemma}
\begin{proof}
Write
\[
\IP\Big[ \|Z\|_{Y^n}\le t\frac{\vol_n(B_X^n)^{1/n}}{\vol_n(B_Y^n)^{1/n}}\Big]
= \IP\Big[ c_n\Big(\|Z\|_{Y^n}\frac{b_n}{a_n}-c\Big) \le c\, e_n c_n\Big].
\]
Taking limits and using the assumed continuity of the distribution function of the random variable $R$ gives the result. 
\end{proof}

\begin{remark}
For completeness we mention that the final conclusion in the previous proof relies on the following fact. For distribution functions $F,F_1,F_2,\ldots$ and real numbers $a,a_1,a_2,\dots$ we have that
\[
\lim_{n\to\infty}F_{n}(a_n)= F(a),
\]
whenever $F_n\to F$ pointwise, $F$ is continuous and $a_n\to a$, as $n\to\infty$. This follows from the implied uniform convergence $F_n\to F$, which is a classical result due to Polya, see~\cite[Satz I]{Pol20}.
\end{remark}

Our arguments below will also rely on the following consequence of Slutsky's theorem.

\begin{lemma} \label{lem:domination}
For $n\in\N$ let $A_n,B_n$ be random variables and $(a_n)_{n\in\N},(b_n)_{n\in\N}$ be positive sequences such that $a_n,b_n\to 0$ as $n\to\infty$. Suppose that $B_n\neq 0$ almost surely for all $n\in\IN$, $a_n/b_n\to\infty$ and assume that there are random variables $A,B$ and real numbers $a,b$ with $b\neq 0$ such that 
\[
a_n^{-1}(A_n-a)\to A\qquad \text{and}\qquad b_n^{-1}(B_n-b)\to B\qquad\text{as }n\to\infty
\]
hold in distribution. Then we have the following convergence in distribution:
\[
a_n^{-1}\Big(\frac{A_n}{B_n}-\frac{a}{b}\Big)\to b^{-1}A\qquad\text{as }n\to\infty.
\]
\end{lemma}

\begin{proof}
Write
\[
a_n^{-1}\Big(\frac{A_n}{B_n}-\frac{a}{b}\Big)
=\frac{a_n^{-1}(A_n-a)b+a(b_n/a_n)b_n^{-1}(b-B_n)}{B_nb}.
\]
By Slutsky's theorem the numerator tends to $bA$ in distribution and since $B_n\to b$ in probability, the denominator tends to $b^2$ in probability. Another application of Slutsky's theorem gives the result.
\end{proof}

\section{Proof of Theorem \ref{thm:pball}}\label{sec:lp}

To show the first part of Theorem \ref{thm:pball} our goal is to apply Lemma \ref{lem:non-crit}. The required distributional limit theorem corresponds to Theorem 1.1(c) in \cite{KPT19}. It says that if $Z$ is uniformly distributed in $\bp$, then
\[
\frac{n^{1/p}}{(p\log n)^{1/p-1}}\|Z\|_{\infty}-A_n^{(p)}\to\Lambda
\]
in distribution as $n\to\infty$, where
\[
A_n^{(p)}=p\log n-\frac{1-p}{p}\log(p\log n)+p\log K_p \qquad \text{with} \qquad K_p=\frac{1}{p^{1/p}\Gamma(1+1/p)}
\]
and where $\Lambda$ is a Gumbel random variable with distribution function $F_{\Lambda}(t)=e^{-e^{-t}}, t\in\IR$. In particular, it holds that
\[
\|Z\|_{\infty}n^{1/p}\to +\infty,
\]
in probability as $n\to\infty$. Moreover, we have that
\[
\lim_{n\to\infty}n^{1/p}\vol_n(\bp)^{1/n}= c_p:=2e^{1/p}p^{1/p}\Gamma(1+1/p) 
\qquad \text{and}\qquad
\vol_n(\binf)^{1/n}=2.
\]
This brings us into the position to apply Lemma~\ref{lem:non-crit} with $a_n=n^{1/p}$, $b_n=1$, $c_X=c_p, c_Y=2$ and $c=+\infty$. This proves the first part of the theorem.

To deduce the second part of Theorem \ref{thm:pball}, fix $t>0$ and again let $Z$ be uniformly distributed in $\bp$. Using computations similar to the above, we arrive at
\[
\vol_n(\Dp \cap t(\log n)^{1/p} \Dinf)
= \IP\Big[\frac{n^{1/p}}{(p\log n)^{1/p-1}}\|Z\|_{\infty}-A_{n}^{(p)}\le R_n^{(p)}\Big]
\]
with
\[
R_n^{(p)}:=\frac{p\,t\,n^{1/p}\vol_n(\bp)^{1/n}}{2p^{1/p}}\log n-A_{n}^{(p)}.
\]
Next, we write
\[
n^{1/p}\vol_n(\bp)^{1/n}=c_p(1+e_n),
\]
where $e_n=-\frac{\log n}{2n}+O(n^{-1})$, see \cite[page 194]{Sch01}. Then
\[
R_n^{(p)}=\Big(\frac{p\,t\, c_p}{2p^{1/p}}(1+e_n)-p+\frac{1-p}{p}\frac{\log(p\log n)}{\log n}-\frac{p\log K_p}{\log n}\Big)\log n.
\]
Setting
\[
\tc_{p,\infty}:=\frac{2p^{1/p}}{c_p}
=e^{-1/p}\Gamma(1+1/p)^{-1},
\]
we deduce that
\[
R_n^{(p)}\to 
\begin{cases}
	-\infty &\colon t<\tc_{p,\infty},\\
	-\infty &\colon t=\tc_{p,\infty},p>1\\
	-\log K_1 &\colon t=\tc_{1,\infty},p=1\\
	+\infty &\colon t>\tc_{p,\infty},
\end{cases}
\]
as $n\to\infty$. Noting that $K_1=1$ and $\IP[\Lambda\le 0]=e^{-1}$, the second part of Theorem~\ref{thm:pball} follows from Lemma~\ref{lem:crit}. \qed

\section{Proof of Theorem \ref{thm:schatten-crit}}\label{sec:Schatten}

We will carry out the proof for arbitrary $\beta$ as far as possible and specialize to the particular case $\beta=2$ only if strictly needed. Before we prove Theorem~\ref{thm:schatten-crit}, we collect some facts about the volume of Schatten $p$-balls. First, observe that the dimension $d_n$ of $\bpb$ defined by \eqref{eq:dn} satisfies $d_n\sim \frac{\beta}{2}n^2$. In general and as explained in \cite[Section 3]{KPT20} or \cite{Sai84} the volume of $\bpb$  can be represented as
\[
\vol_{\beta,n}(\bpb)
=c_{n,\beta}I_{n,\beta,p},
\]
where
\[
c_{n,\beta}=\frac{1}{n!}\Big(\frac{2\pi^{\beta/2}}{\Gamma(\frac{\beta}{2})}\Big)^{-n}\prod_{k=1}^{n}\frac{2(2\pi)^{\beta k/2}}{2^{\beta/2}\Gamma(\frac{\beta k}{2})}
\qquad\text{and}\qquad
I_{n,\beta,p}=\int_{\bp}\prod_{1\le i<j\le n}|\lambda_j-\lambda_i|^{\beta}\dd \lambda_1\dots\dd\lambda_n.
\]
The integral $I_{n,\beta,p}$ can be computed exactly only in the cases $p=2$ and $p=\infty$ using integral formulas due to Mehta and Selberg, respectively, see \cite[Theorem 2.5.8]{AGZ10}. For $p=2$ we have $\btb=\mathbb{B}_2^{d_n}$ and thus
\[
\vol_{\beta,n}(\btb)
=\vol_{\beta,n}(\mathbb{B}_2^{d_n})
=\frac{\pi^{d_n/2}}{\Gamma(1+d_n/2)}.
\]
This is due to the fact that the Schatten norm becomes the Frobenius norm in this case, which in turn is the usual Euclidean norm after interpreting a matrix as a vector. Furthermore,
\[
\vol_{\beta,n}(\binfb)
=c_{n,\beta}2^{d_n}n!\prod_{j=0}^{n-1}\frac{\Gamma\left( 1+j\frac{\beta}{2} \right)^2\Gamma\left( (j+1)\frac{\beta}{2} \right)}{\Gamma\left( 2+(n+j-1)\frac{\beta}{2} \right)\Gamma\left( \frac{\beta}{2} \right)}.
\]
By refining the method used by Saint-Raymond in \cite{Sai84}, the asymptotic volume radius of $\bpb$ has been obtained in \cite[Theorem 3.1]{KPT20}, which says that
\begin{equation} \label{eq:vol-asymp}
\vol_{\beta,n}(\bpb)^{1/d_n}
\sim n^{-(1/p+1/2)}\Delta(p)\Big(\frac{4\pi}{\beta}\Big)^{1/2}e^{3/4},
\end{equation}
where
\[
\Delta(p) = \begin{cases}
\frac{1}{2}\Big(\frac{p\sqrt{\pi}\Gamma(p/2)}{\sqrt{e}\Gamma( (p+1)/2)}\Big)^{1/p} &\colon p<\infty\\
\frac{1}{2}&\colon p=\infty.
\end{cases}
\]
In particular, $\Delta(2)=e^{-1/4}$ since $\Gamma(3/2)=\sqrt{\pi}/2$. For our purposes we need the following fine asymptotics of the logarithmic volume of $\bpb$, which we are able to provide only in the two special cases $p=2$ and $p=\infty$.

\begin{lemma} \label{lem:log-vol}
	For $\beta\in\{1,2,4\}$ we have that
	\[
	\frac{1}{d_n}\log \vol_{\beta,n}(\btb)
	=-\log n + {1\over 2}\log\Big({4\pi\over\beta}\Big)+{1\over 2}+ O(n^{-1}\log n)
	\]
	and
	\[
	\frac{1}{d_n}\log \vol_{\beta,n}(\binfb)
	=-{1\over 2}\log n + {1\over 2}\log\Big({4\pi\over\beta}\Big)+{3\over 4} - \log 2 +  O(n^{-1}\log n).
	\]
\end{lemma}
\begin{proof}
We start with the case $p=2$ and by recalling \cite[Identity 6.1.41]{Abramowitz}, which says that
\begin{align}\label{eq:LogGammaAsymptotic}
\log \Gamma(z)=(z-{1\over 2})\log z-z+{1\over 2}\log(2\pi)+O(z^{-1}),\qquad\text{as }z\to\infty.
\end{align}
This together with the fact that $d_n\sim{\beta n^2\over 2}$ leads to
\begin{align*}	
{1\over d_n}\log \vol_{\beta,n}(\btb) &= {1\over 2}\log \pi - {1\over d_n}\log\Gamma\Big(1+{d_n\over 2}\Big)\\
&={1\over 2}\log \pi - {1\over 2}\log{d_n\over 2}+{1\over d_n}{d_n\over 2} + O(n^{-1}\log n)\\
&=-\log n + {1\over 2}\log\Big({4\pi\over\beta}\Big)+{1\over 2}+ O(n^{-1}\log n),
\end{align*}
as claimed. 

For $p=\infty$ we start by observing that
\begin{align*}
{1\over d_n}\log \vol_{\beta,n}(\binfb) &= {1\over d_n}\log c_{n,\beta} + \log 2 + {1\over d_n}\log n! + S_n 
\end{align*}
with
$$
S_n := {1\over d_n}\sum_{j=0}^{n-1} \Big[2\log\Gamma\Big(1+{\beta j\over 2}\Big) + \log\Gamma\Big((j+1){\beta\over 2}\Big)-\log\Gamma\Big(2+(n+j-1){\beta\over 2}\Big)-\log\Gamma\Big({\beta\over 2}\Big)\Big].
$$
The proof of  \cite[Lemma 3.3]{KPT20} directly shows that
$$
{1\over d_n}\log c_{n,\beta} = -{1\over 2}\log n + {1\over 2}\log\Big({4\pi\over\beta}\Big)+{3\over 4}+O(n^{-1}\log n).
$$
Moreover, ${1\over d_n}\log n! = O(n^{-1}\log n)$ by \eqref{eq:LogGammaAsymptotic} and the fact that $d_n\sim{\beta n^2\over 2}$. Thus, it remains to determine the asymptotic behavior of $S_n$. First, we observe that
$$
S_n = {1\over d_n}\sum_{j=0}^{n-1} \Big[2\log\Gamma\Big(1+{\beta j\over 2}\Big) + \log\Gamma\Big((j+1){\beta\over 2}\Big)-\log\Gamma\Big(2+(n+j-1){\beta\over 2}\Big)\Big] + O(n^{-1}),
$$
since $d_n\sim{\beta n^2\over 2}$. Now, we use three times \eqref{eq:LogGammaAsymptotic} to see that
\begin{align*}
S_n={1\over d_n}\sum_{j=0}^{n-1}\Big[{3\beta\over 2}j\log\Big({\beta j\over 2}\Big)-{3\beta\over 2} j-{\beta\over 2}(n+j)\log\Big({\beta\over 2}(n+j)\Big)+{\beta\over 2}(n+j)\Big]+O(n^{-1}\log n)
\end{align*}
To evaluate these sums asymptotically, we use the following formula, which follows from Abel's summation formula as in the proof of \cite[Lemma 3.3]{KPT20}:
\begin{align}
&\sum_{j=0}^{n-1}(an+bj)\log(an+bj)\nonumber\\
&= \Big(a+{b\over 2}\Big)n^2\log n + {n^2\over 4b}\big(2(a+b)^2\log(a+b)-2a^2\log a-b(2a+b)\big)+O(n\log n)
\end{align}
as $n\to\infty$ for constants $a,b\geq 0$, where for $a=0$ or $b=0$ we use the convention $0^2\log 0=0$. Applying this with the choices $(a,b)=(0,{\beta\over 2})$ and $(a,b)=({\beta\over 2},{\beta\over 2})$ together with the fact that $\sum_{j=0}^{n-1}j={n^2\over 2}+O(n)$, we arrive at
\begin{align*}
S_n &= {1\over d_n}\Big[{3\beta\over 4}n^2\log n-{3\beta\over 8}\big(1+2\log 2-2\log\beta\big)n^2-{3\beta\over 4}n^2\\
&\qquad\qquad-{3\beta\over 4}n^2\log n-{\beta\over 8}\big(2\log 2+6\log\beta-3\big)n^2+{3\beta\over 4}n^2\Big] + O(n^{-1}\log n)\\
&=-{1\over d_n} n^2\beta\log 2 + O(n^{-1}\log n)\\
&=-2\log 2+ O(n^{-1}\log n),
\end{align*}
using additionally that $d_n\sim{\beta n^2\over 2}$.  

Summarizing, this leads to
\begin{align*}
{1\over d_n}\log \vol_{\beta,n}(\binfb) &= -{1\over 2}\log n + {1\over 2}\log\Big({4\pi\over\beta}\Big)+{3\over 4} + \log 2 -2\log 2+ O(n^{-1}\log n)\\
&= -{1\over 2}\log n + {1\over 2}\log\Big({4\pi\over\beta}\Big)+{3\over 4} -  \log 2 +  O(n^{-1}\log n),
\end{align*}
which proves the second claim.
\end{proof}

\begin{remark}
Interpreting $1/\infty$ as $0$ and using that $\log\Delta(2)=-1/4$ and $\log\Delta(\infty)=-\log 2$, the two formulas in Lemma \ref{lem:log-vol} can be combined as follows:
$$
{1\over d_n}\log\vol_{\beta,n}\bpb = -\Big({1\over 2}+{1\over p}\Big)\log n+{1\over 2}\log\Big({4\pi\over\beta}\Big)+{3\over 4}+\log\Delta(p)+  O(n^{-1}\log n),
$$
$p\in\{2,\infty\}$. It remains an open problem if this formula continues to hold for arbitrary $p\geq 1$. As we pointed out in the introduction, the method used in \cite{KPT20} does not seem to be suitable to yield a result of this quality.
\end{remark}

As a consequence, we obtain an upper bound on the speed of convergence of the ratio of volume radii of $\btb$ and $\binfb$ as required in \eqref{eq:AsymVolRadius}.

\begin{lemma} \label{cor:ratio-asymp}
	For $\beta\in\{1,2,4\}$ it holds that
\[
\frac{n^{1/2}\vol_{\beta,n}(\btb)^{1/d_n}}{\vol_{\beta,n}(\binfb)^{1/d_n}}
=\frac{\Delta(2)}{\Delta(\infty)}\big(1+O(n^{-1}\log n)\big).
\]
\end{lemma}
\begin{proof}
Write
\[
T_n:=\frac{n^{1/2}\vol_{\beta,n}(\btb)^{1/d_n}}{\vol_{\beta,n}(\binfb)^{1/d_n}}=\frac{\Delta(2)}{\Delta(\infty)}(1+e_n),
\]
where $e_n\to 0$ by the asymptotics in \eqref{eq:vol-asymp}. Because of $\log(1+e_n)= e_n+O(e_n^2)$ as $n\to\infty$, it follows that
\[
\log T_n
=\frac{1}{2}\log n+\frac{1}{d_n}\log\Big(\frac{\vol_{\beta,n}(\btb)}{\vol_{\beta,n}(\binfb)}\Big)
=\log\Delta(2)-\log\Delta(\infty)+e_n+O(e_n^2).
\]
Now, Lemma~\ref{lem:log-vol} yields $e_n=O(n^{-1}\log n)$, which completes the proof.
\end{proof}

In order to prove Theorem \ref{thm:schatten-crit}, we will use Lemma~\ref{lem:crit}. To apply it, we need the volume asymptotics from Lemma~\ref{cor:ratio-asymp} and also a limit theorem for the norm of a random vector uniformly distributed in $\btt$. It is based on the following stochastic representation taken from \cite[Corollary 4.3]{KPT20}.

\begin{lemma}\label{pro:stoch-rep}
If $Z$ is uniformly distributed in $\bpb$ for $1\le p<\infty$, then
\[
(\lambda_{\pi(1)}(Z),\dots,\lambda_{\pi(n)}(Z))
\overset{\rm d}{=} U^{1/d_n}\frac{X}{\|X\|_p},
\]
where $U$ is uniformly distributed on $[0,1]$, $X$ is independent of $U$ and has joint density with respect to Lebesgue measure on $\IR^n$ proportional to 
\begin{equation}\label{eq:fpbetan}
f_{p,\beta,n}(x)
=e^{-\sum_{j=1}^{n}|x_i|^p}\prod_{1\le i<j\le n}|x_i-x_j|^{\beta},\qquad x\in\IR^n,
\end{equation}
and $\pi$ is a uniformly distributed random permutation on $\{1,\dots,n\}$. 
\end{lemma}

Using this stochastic representation and known limit laws for spectral statistics of the Gaussian Unitary Ensemble (GUE), we can prove the following limit law.

\begin{lemma}\label{pro:limit-norm}
Let $Z$ be uniformly distributed on $\btt$. For every $x\in\IR$, it holds that
\[
\lim_{n\to\infty}\IP[\sqrt{2}n^{2/3}(n^{1/2}\|Z\|_{\infty}-2)\le x]
= F_{2}(2^{-1/2}x)^2.
\]
\end{lemma}

Before proving Lemma~\ref{pro:limit-norm}, let us give a proof of Theorem~\ref{thm:schatten-crit}.

\begin{proof}[Proof of Theorem~\ref{thm:schatten-crit}]
We start by writing
\[
\frac{n\vol_{\beta,n}(\btb)^{1/d_n}}{n^{1/2}\vol_{\beta,n}(\binfb)^{1/d_n}}
=\frac{\Delta(2)}{\Delta(\infty)}(1+e_n).
\]
Together with Lemma~\ref{pro:limit-norm} this implies that $\tc_{2,\infty,2}=2\frac{\Delta(\infty)}{\Delta(2)}=e^{1/4}$, and with Lemma~\ref{lem:non-crit} we obtain the non-critical case. By Lemma~\ref{cor:ratio-asymp} we have $e_n=O(n^{-1}\log n)$ and consequently $e_n n^{2/3}\to 0$ as $n\to\infty$. We can now deduce the critical case from Lemma~\ref{lem:crit} together with Lemma~\ref{pro:limit-norm}.
\end{proof}

In the remainder of this section we give the proof of Lemma~\ref{pro:limit-norm}. Recall that by Lemma~\ref{pro:stoch-rep} we have for $Z$ uniformly distributed in $\btt$ that
\begin{equation}\label{eq:170823A}
n^{1/2}\|Z\|_{\infty}
=n^{1/2}\max|\lambda_i(Z)|
=U^{1/d_n}\frac{n^{-1/2}\max|X_i|}{\Big(\frac{1}{n}\sum_{i=1}^{n}|n^{-1/2}X_i|^2\Big)^{1/2}},
\end{equation}
where $X$ is distributed with density proportional to $f_{2,2,n}$, recall \eqref{eq:fpbetan}. We will first derive limit theorems for numerator and denominator of \eqref{eq:170823A} separately and then combine them using Lemma~\ref{lem:domination}.

\begin{lemma} \label{lem:abs-tracy}
If $X$ is distributed with density proportional to $f_{2,2,n}$, then 
\[
\lim_{n\to\infty}\IP[\sqrt{2}n^{2/3}(n^{-1/2}\max|X_i|-\sqrt{2})\le x]= F_{2}(x)^2, \qquad x\in\IR.
\]
\end{lemma}
\begin{proof}
We put 
\[
X_{\beta,n}^{\min}:=\sqrt{2}n^{2/3}(n^{-1/2}\min X_i +\sqrt{2})
\quad\text{and}\quad
X_{\beta,n}^{\max}:=\sqrt{2}n^{2/3}(n^{-1/2}\max X_i -\sqrt{2}).
\]
The fluctuations for the largest eigenvalue were stated in the introduction and, by symmetry of the density function, the random variables $\max X_i$ and $-\min X_i$ are identically distributed. Thus, we obtain that, as $n\to\infty$,
\[
\IP[X_{\beta,n}^{\max}\le x]\to F_{\beta}(x)
\qquad\text{and}\qquad
\IP[X_{\beta,n}^{\min}\le x]\to 1-F_{\beta}(-x), \qquad x\in\IR.
\]
By the asymptotic independence stated in \cite[Corollary 1]{BDN10} and mentioned above in Conjecture~\ref{conjecture} we deduce that, for every $x\in\IR$,
\begin{align*}
	\IP[\sqrt{2}n^{2/3}(n^{-1/2}\max |X_i|-\sqrt{2})\le x]
	&=\IP[X_{\beta,n}^{\max}\le x,X_{\beta,n}^{\min}\ge -x]\\
	&=\IP[X_{\beta,n}^{\max}\le x]-\IP[X_{\beta,n}^{\max}\le x,X_{\beta,n}^{\min}\le -x]\\
	&\to F_2(x)-F_2(x)(1-F_2(x))\\
	&=F_2(x)^2,
\end{align*}
as $n\to\infty$.
\end{proof}

For the numerator in \eqref{eq:170823A} a central limit theorem holds, which we directly derive from the well-known fluctuations of linear spectral statistics. The relevant case $\beta=2$ was already proved by Johansson in \cite{Joh98} but we will use the more general result from \cite{LLW19}.

\begin{lemma} \label{lem:numerator}
For $\beta\in\{1,2,4\}$ it holds that
\[
n\Big(\Big(\frac{1}{n}\sum_{i=1}^{n}|(\beta n)^{-1/2}X_i|^2\Big)^{1/2}-\frac{1}{2}\Big)\to N
\]
in distribution as $n\to\infty$, where $N$ is a Gaussian random variable with mean $\mu=(\frac{1}{2}-\frac{1}{\beta})\frac{3}{4}$ and some non-degenerate variance $\sigma^2>0$.
\end{lemma}
\begin{proof}
	The random variables $(\beta n)^{-1/2}X_1,\ldots,(\beta n)^{-1/2}X_n$ are distributed according to the density $P_{V,\beta}^n$ as denoted in \cite{LLW19} where the $V\colon x\mapsto x^2$. The limiting measure is given by the Ullmann distribution $\mathcal{U}(2)$ and if $W\sim \mathcal{U}(2)$, then $\IE |W|^2=1/4$, see \cite[IV,Theorem 5.1]{ST97}. Therefore, by \cite[Theorem 1.2]{LLW19} we get
\[
n\Big(\frac{1}{n}\sum_{i=1}^{n}|(\beta n)^{-1/2}X_i|^2-\frac{1}{4}\Big)\to N
\]
in distribution as $n\to\infty$. We now apply the delta method with the function $g(x)=x^{1/2}$ at the point $x_0=\frac{1}{4}$ giving $g'(x_0)=1$ to deduce the result.
\end{proof}

\begin{proof}[Proof of Lemma~\ref{pro:limit-norm}]
Applying the stochastic representation of Lemma \ref{pro:stoch-rep} for $\beta=2$ we have
\[
n^{1/2}\|Z\|_{\infty}
=U^{1/d_n}\frac{n^{-1/2}\max|X_i|}{\sqrt{2}\Big(\frac{1}{n}\sum_{i=1}^{n}|(2 n)^{-1/2}X_i|^2\Big)^{1/2}}
\]
which is of the form $U^{1/d_n}\frac{A_n}{B_n}$. We deduce from Lemma \ref{lem:abs-tracy} and Lemma \ref{lem:numerator} that
\[
\sqrt{2}n^{2/3}(A_n-\sqrt{2})\to A
\qquad\text{and}\qquad
n(B_n-2^{-1/2})\to B,
\]
as $n\to\infty$ in distribution, where $A$ has distribution function $F_2^2$ and $B$ is a Gaussian random variable with mean $\sqrt{2}\mu$ and variance $2\sigma^2$.

As a consequence of Lemma~\ref{lem:domination}, applied with $a_n=2^{-1/2}n^{-2/3}, b_n=n^{-1},a=\sqrt{2},b=2^{-1/2}$ and in combination with Lemma \ref{lem:abs-tracy} and Lemma \ref{lem:numerator}, we conclude that for all $x\in\IR$,
\[
\lim_{n\to\infty}\IP\Big[\sqrt{2}n^{2/3}\Big(\frac{A_n}{B_n}-2\Big)\le x\Big]
= \IP[A\le 2^{-1/2}x]
=F_{2}(2^{-1/2}x)^2.
\]
The proof can now be concluded by an application of Slutsky's theorem, since $U^{1/d_n}$ converges in probability to $1$, as $n\to\infty$.
\end{proof}

\subsection*{Acknowledgement}
The authors thank Djalil Chafa\"i for providing references to asymptotic independence of the fluctuations of the extremal eigenvalues of random matrices.\\ MS has been supported by the Austrian Science Fund (FWF) Project P32405 \textit{Asymptotic geometric analysis and applications}. CT has been supported by the DFG via SFB/TR 191 and the project \textit{Limit theorems for the volume of random projections of $\ell_p$-balls} (project number 516672205). 

\bibliographystyle{plain}
\bibliography{critinter}

\begin{thebibliography}{10}

\bibitem{Abramowitz}
M.~Abramowitz and I.~A. Stegun, editors.
\newblock {\em Handbook of {M}athematical {F}unctions with {F}ormulas,
  {G}raphs, and {M}athematical {T}ables}.
\newblock Dover Publications, Inc., New York, 1992.

\bibitem{AGZ10}
G.~W. Anderson, A.~Guionnet, and O.~Zeitouni.
\newblock {\em An {I}ntroduction to {R}andom {M}atrices}, volume 118 of {\em
  Cambridge Studies in Advanced Mathematics}.
\newblock Cambridge University Press, Cambridge, 2010.

\bibitem{BKP+20}
A.~Baci, Z.~Kabluchko, J.~Prochno, M.~Sonnleitner, and C.~Th\"{a}le.
\newblock Limit theorems for random points in a simplex.
\newblock {\em J. Appl. Probab.}, 59(3):685--701, 2022.

\bibitem{BDN10}
P.~Bianchi, M.~Debbah, and J.~Najim.
\newblock Asymptotic independence in the spectrum of the {G}aussian unitary
  ensemble.
\newblock {\em Electron. Commun. Probab.}, 15:376--395, 2010.

\bibitem{Bor10}
F.~Bornemann.
\newblock Asymptotic independence of the extreme eigenvalues of {G}aussian
  unitary ensemble.
\newblock {\em J. Math. Phys.}, 51(2):023514, 8, 2010.

\bibitem{JL15}
T.~Jiang and D.~Li.
\newblock Approximation of rectangular beta-{L}aguerre ensembles and large
  deviations.
\newblock {\em J. Theoret. Probab.}, 28(3):804--847, 2015.

\bibitem{Joh98}
K.~Johansson.
\newblock On fluctuations of eigenvalues of random {H}ermitian matrices.
\newblock {\em Duke Math. J.}, 91(1):151--204, 1998.

\bibitem{JKP22}
M.~{Juhos}, Z.~{Kabluchko}, and J.~{Prochno}.
\newblock {Limit theorems for mixed-norm sequence spaces with applications to
  volume distribution}.
\newblock {\em Electron. J. Probab.}, 29:1--44, 2024.

\bibitem{JP21}
M.~Juhos and J.~Prochno.
\newblock Spectral flatness and the volume of intersections of $p$-ellipsoids.
\newblock {\em J. Complexity}, 70:101617, 2022.

\bibitem{KP21}
Z.~Kabluchko and J.~Prochno.
\newblock The maximum entropy principle and volumetric properties of {O}rlicz
  balls.
\newblock {\em J. Math. Anal. Appl.}, 495(1):Paper No. 124687, 19, 2021.

\bibitem{KPS23}
Z.~{Kabluchko}, J.~{Prochno}, and M.~Sonnleitner.
\newblock {A probabilistic approach to Lorentz balls}.
\newblock {\em J. Funct. Anal.}, 288(1):110682, 2025.

\bibitem{KPT19}
Z.~Kabluchko, J.~Prochno, and C.~Th\"{a}le.
\newblock High-dimensional limit theorems for random vectors in
  {$\ell_p^n$}-balls.
\newblock {\em Commun. Contemp. Math.}, 21(1):1750092, 30, 2019.

\bibitem{KPT20}
Z.~Kabluchko, J.~Prochno, and C.~Th\"{a}le.
\newblock Intersection of unit balls in classical matrix ensembles.
\newblock {\em Israel J. Math.}, 239(1):129--172, 2020.

\bibitem{LLW19}
G.~Lambert, M.~Ledoux, and C.~Webb.
\newblock Quantitative normal approximation of linear statistics of
  {$\beta$}-ensembles.
\newblock {\em Ann. Probab.}, 47(5):2619--2685, 2019.

\bibitem{Pol20}
G.~P\'{o}lya.
\newblock \"{U}ber den zentralen {G}renzwertsatz der
  {W}ahrscheinlichkeitsrechnung und das {M}omentenproblem.
\newblock {\em Math. Z.}, 8(3-4):171--181, 1920.

\bibitem{PTT19}
J.~Prochno, C.~Th\"{a}le, and N.~Turchi.
\newblock Geometry of {$\ell^n_p$}-balls: classical results and recent
  developments.
\newblock In {\em High dimensional probability {VIII}---the {O}axaca volume},
  volume~74 of {\em Progr. Probab.}, pages 121--150. Birkh\"{a}user/Springer,
  Cham, 2019.

\bibitem{RR91}
S.T. Rachev and L.~R{\"u}schendorf.
\newblock Approximate independence of distributions on spheres and their
  stability properties.
\newblock {\em Ann. Probab.}, 19(3):1311--1337, 1991.

\bibitem{ST97}
E.~B. Saff and V.~Totik.
\newblock {\em Logarithmic {P}otentials with {E}xternal {F}ields}, volume 316
  of {\em Grundlehren der mathematischen Wissenschaften}.
\newblock Springer-Verlag, Berlin, 1997.

\bibitem{Sai84}
J.~Saint-Raymond.
\newblock Le volume des id\'{e}aux d'op\'{e}rateurs classiques.
\newblock {\em Studia Math.}, 80(1):63--75, 1984.

\bibitem{SS91}
G.~Schechtman and M.~Schmuckenschl\"{a}ger.
\newblock Another remark on the volume of the intersection of two {$L^n_p$}
  balls.
\newblock In {\em Geometric aspects of functional analysis (1989--90)}, volume
  1469 of {\em Lecture Notes in Math.}, pages 174--178. Springer, Berlin, 1991.

\bibitem{SZ90}
G.~Schechtman and J.~Zinn.
\newblock On the volume of the intersection of two {$L^n_p$} balls.
\newblock {\em Proc. Amer. Math. Soc.}, 110(1):217--224, 1990.

\bibitem{Sch98}
M.~Schmuckenschl\"{a}ger.
\newblock Volume of intersections and sections of the unit ball of {$l^n_p$}.
\newblock {\em Proc. Amer. Math. Soc.}, 126(5):1527--1530, 1998.

\bibitem{Sch01}
M.~Schmuckenschl\"{a}ger.
\newblock C{LT} and the volume of intersections of {$l^n_p$}-balls.
\newblock {\em Geom. Dedicata}, 85(1-3):189--195, 2001.

\end{thebibliography}

\end{document}